\documentclass[10pt,a4paper]{article}
\usepackage{graphicx}
\usepackage{amsfonts}
\usepackage{amsmath}
\usepackage{amssymb}
\usepackage{amsthm}
\usepackage{times}
\usepackage{bm}
\usepackage{natbib}

\newtheorem{thm}{Theorem}[section]

\begin{document}

\title{Another argument in favour of Wilcoxon's signed rank test}

\author{
Jonathan D. Rosenblatt\\
\texttt{jonathan.rosenblatt@weizmann.ac.il} \\\\ 
The Weizmann Institute of Science
\and
Yoav Benjamini\\
\texttt{ybenja@post.tau.ac.il}\\\\
Tel Aviv University}

\maketitle

\begin{abstract}

The Wilcoxon Signed Rank test is typically called upon when testing whether a symmetric  distribution has a specified centre and the Gaussianity is in question. As with all insurance policies it comes with a cost, even if small, in terms of power versus a t-test, when the distribution is indeed Gaussian.
In this note we further show that even when the distribution tested is Gaussian there need not be power loss at all, if the alternative is of a mixture type rather than a shift. The signed rank test may turn out to be more powerful than the t-test, and the supposedly conservative strategy, might actually be the more powerful one. Drug testing and functional magnetic imaging are two such scenarios.


\end{abstract}

\section{Introduction}
Consider a testing whether the centre of a symmetric distribution is at a specified location, say zero, using a single sample.
The two commonly used test statistics are the t-statistic, defined as $T =\bar{X}/ \left( S/\surd{n}\right)$, and the Wilcoxon's Signed-Rank statistic, defined as
$W^{+}=\sum_{i=1}^{n}R_{i}^{+}1_{ \left\{ X_{i}>0 \right\} }$,
where $R_{i}^{+}$ is the rank of the absolute value of the $i$th observation.
The signed-rank test statistic that does not rely on the Gaussian assumption for the distribution was historically assumed to suffer
from low power versus the t-statistic.
Numerical results, and the concept of asymptotic relative efficiency (abbreviated as
efficiency herein for simplicity) allowed to quantify this power loss.
They taught us that the signed-rank statistic is actually more powerful in many realistic scenarios \citep{lehmann_parametric_2009},
in particular when the test statistic is distributed with heavier-than-Gaussian-tails.
Of course, the t-statistic is optimal when the symmetric distribution is Gaussian,
though only 1.05 times more efficient than the signed-rank test.
All of these results consider the efficiency (and optimality) under alternatives
in which the distributional shape does not change, and only the location is shifted.
This type of alternatives need not be the true type of deviation from the null.


Say the true deviation from the null is actually of a mixture type. Is the researcher really loosing power when opting for the signed-rank test?
An example arises when localizing cognitive regions in whole-brain functional-magnetic-resonance-imaging
group studies; The researchers have to register all scanned brains
to a common spatial template. Since registration is imperfect, and
since function and anatomy are not one-to-one, any given brain location
might include both active and inactive individuals.
If one now defines an active location as non null mean activation
the mixture alternative is a more natural formulation than the shift alternative. This is indeed the approach adopted in  a recently submitted paper by Rosenblatt, Vink and Benjamini (2012), which motivated this note.
In that particular study, the signed-rank test statistic allowed
for more detections than the regularly-used t statistic: 11,817 out
of 27,401 brain locations were declared active using the former versus
11,037 using the latter. 

The comparison of these two test statistics for mixture alternatives
using the concept of relative asymptotic efficiency requires no new tools, but yields some surprising results:
A researcher reluctant to assume Gaussianity, opting for the signed-rank test, might actually be giving up less power than he thought. He might even be gaining some. In fact, the preferred statistic depends on the nuisance parameters in a way that prefers the signed-rank statistic in many realistic scenarios.

\section{Result}

Denote $\mathcal{N}\left(\mu,\sigma^{2}\right)$ to be the Gaussian
distribution with mean $\mu$ and variance $\sigma^{2}$ and consider
a simple random sample $X_{1},...,X_{n}$ from 
\begin{equation}
\left(1-\theta\right)\mathcal{N}\left(0,1\right)+\theta\mathcal{N}\left(\mu,\sigma^{2}\right)\label{eq:mixture_pdf}
\end{equation}
We now wish to compare the test statistics to detect $H_{1}:\theta\in(0,1]$
versus $H_{0}:\theta=0$ when $\mu,\sigma^{2}$ are assumed known. 

Let $n_{i}\left(\alpha,\pi,\theta\right)$ be the minimal sample size
for a test statistic $T_{i}$ to achieve power no smaller than $\pi$
with type I error rate no greater than $\alpha$. The Pitman asymptotic
relative efficiency between any two test statistics in our setup is defined as
\[
e_{1,2} = \underset{\theta\downarrow0}{\lim} \left\{ \frac{n_{2} \left( \alpha,\pi,\theta \right) }{n_{1} \left( \alpha,\pi,\theta \right) } \right\}
\]

The main result of this note is the following: 
\begin{thm}
Given a random sample from a population distributed as in eq.~\ref{eq:mixture_pdf},
	when testing the hypotheses $H_{1}:\theta\in(0,1]$ versus $H_{0}:\theta=0$,
	 the asymptotic relative efficiency of the Wilcoxon Signed-Rank versus
	the T-statistic statistic is: 
	\begin{equation}
	e_{Wilcoxon,T} =
	 \frac{9}{\mu^{2}} \left[ 2\Phi \left\{ \frac{\mu}{ \left( 1+\sigma^{2} \right) ^{1/2}} \right\} -1 \right] ^{2} \label{eq:main-result}
	\end{equation}
	\label{theorem1}
\end{thm}

\begin{proof}
To prove this we follow the lines of Chapter 14 of \citet{vaart_asymptotic_2000}:
We first show that the two test statistics are locally asymptotically
normal when considering a series of hypotheses that approach the null
at rate $n^{1/2}$ as the sample grows. Namely $\theta_{n}=h/n^{1/2}$
for some $h>0$ . We then note that as $\theta\downarrow0$, the sample
sizes needed for fixed $\alpha$ and $\pi$ grow at the required 
rate for local asymptotic normality to hold. This greatly simplifies
the derivation of the efficiency.

\subsection*{The T-statistic:}
Local asymptotic normality as $n\rightarrow0$ under the series $\theta_{n}=h/n^{1/2}$ is immediate using the central limit theorem under the Lindberg condition: 
$n^{1/2}\left\{\bar{X}/s-\xi\left(\theta_{n}\right)\right\}{\rightarrow}\mathcal{N}\left(0,1\right)$  where $S$ is the root of the unbiased variance estimator and $\xi\left(\theta\right)=\theta\mu/\left\{\left(1-\theta\right)+\theta\sigma^{2}\right\}^{1/2}$
is the mean function. 
The Lindberg condition can be checked directly.

\subsection*{The Wilcoxon signed-rank statistic:}
As noted in page 164 in \citet{vaart_asymptotic_2000} the signed-rank
test is asymptotically equivalent to using the following U statistic:
$U=\binom{n}{2}^{-1}\sum_{i}\sum_{j>i}1_{\left\{ X_{i}+X_{j}>0\right\} }$
since $W^{+}=\binom{n}{2}U+\sum_{i}1_{\left\{ X_{i}>0\right\} }$
and the second term of the left hand side is an order smaller than
the first. Being a $U$ statistic with kernel $h_{2}\left(X_{1},X_{2}\right)=I_{\left\{ x_{1}+x_{2}>0\right\} }$,
implies, by the U-statistic theorem, asymptotic normality for all
fixed $\theta$: $n^{1/2}\left\{U_{n}-\xi\left(\theta\right)\right\}\rightarrow\mathcal{N}\left\{0,r^{2}\zeta_{1}\left(\theta\right)\right\}$
in distribution where $r$ is the order of the $U$ statistic 
,thus $r= 2$ for the signed-rank statistics, $F_{\theta}\left(X\right)$
is the cumulative distribution function of the mixture indexed by
$\theta$ and

\begin{eqnarray*}
\xi \left( \theta \right) & = 
& E_{\theta} \left\{ h_{2} \left( X_{1},X_{2} \right) \right\}= P \left( X_{1}+X_{2}>0 \right) \\ 
 &  & =
 \theta^{2}\Phi \left( \frac{\surd{2}\mu}{\sigma} \right) -\frac{1}{2}(\theta-1)
  \left[ 1-\theta+4\theta\Phi \left\{ \frac{\mu}{ \left( 1+\sigma^{2} \right) ^{1/2}} \right\} \right]
\end{eqnarray*}

\[
\zeta_{1} \left( \theta \right) = 
\text{var}_{\theta} \left[ E_{\theta} \left\{ h_{2} \left( X_{1,}X_{2} \right) \mid X_{1} \right\} \right] = \text{var}_{\theta} \left\{ 1-F_{\theta} \left( X \right) \right\} = \frac{1}{12}
\]

We thus have uniformly for all fixed $\theta$: 
$n^{1/2} \left\{ U_{n} - \xi \left( \theta \right) \right\} \rightarrow\mathcal{N} \left( 0,1/3 \right) $ 
in distribution and in particular for the null $\theta=0$. Local asymptotic normality
can be now established by using Hajek's projection and Lindberg's
central limit Theorem; The Hajek projection of 
$ U_{n} - \xi \left( \theta \right) $
is 
\[
\hat{U}_{n} = 
-\frac{2}{n}\sum_{i=1}^{n} \left[ F \left( -X_{i} \right) -E \left\{ F \left( -X_{i} \right) \right\} \right]
\]
Lindberg's condition trivially holds for a bounded random variable
so by the Central Limit Theorem $\surd{n}\hat{U}$ will converge in distribution to 
$ \mathcal{N} \left( 0,1/3 \right) $ as $\theta_{n} = h/\surd{n} $
approaches $\theta=0$. 
Now since var($U_n$)/var($\hat{U_n}$) converges in probability to 1, it follows from Theorem 11.2 in \citet{vaart_asymptotic_2000} that
$n^{1/2} \left\{ U-\xi \left( \theta \right) - \hat{U} \right\} $ converges
in probability to zero for all $\theta$ . Applying Slutskey's
lemma we get the desired local asymptotic normality: $n^{1/2}\left[U_{n}-\xi\left(\theta\right)\right]{\rightarrow}\mathcal{N} \left( 0,1/3 \right) $ in distribution as $\theta_{n}\rightarrow0$.

\subsection*{From local asymptotic normality to asymptotic relative efficiency:}

We will now use local asymptotic normality to compute the efficiency
between the test statistics.
The joint density of $n$ independent observations each having the mixture density in Eq. (~\ref{eq:mixture_pdf})

Let $P_{n}\left(\theta\right)$ be the probability measure of $n$ independent identically distributed
observations. 
Notice 
$P_{n} \left( \theta \right) $ is smooth in $\theta$ around
$\theta=0$ in the sense that the total variation distance between
$P_{n} \left( \theta \right) $ and $P_{n} \left( 0 \right) $ vanishes as
$\theta\downarrow0$. Also, for all $n$ the power function is non
decreasing in $\theta$. These conditions suffice for the efficiency
to equate the squared efficacy ratio \citep[Theorem 14.19]{vaart_asymptotic_2000}:
\[
e_{1,2} = 
 \left\{ \frac{\xi_{1}' \left( 0 \right) /\sigma_{1} \left( 0 \right) }{\xi_{2}' \left( 0 \right) /\sigma_{2} \left( 0 \right) } \right\} ^ {2}
\]
where $\xi_{i}'\left(0\right)$ is the derivative of $T_{i}$'s mean
function at the null, and $\sigma_{1}\left(0\right)$ is its standard
deviation. 
We now plug in the appropriate mean and variance for the
t and Wilcoxon statistics respectively:
$\xi_{T}' \left( 0 \right) = \mu$, 
$\sigma_{T} \left( 0 \right) = 1$, 
$\xi'_{Wilcoxon} \left( 0 \right) = 2\Phi \left\{ \mu/{ \left( 1+\sigma^{2} \right) ^{1/2}} \right\} - 1$ 
and  
$\sigma_{Wilcoxon} \left( 0 \right) = \left(1/3 \right) ^{1/2}$.
Equation~\ref{eq:main-result} follows.
\end{proof}

\section{Discussion}

As Theorem ~\ref{theorem1} indicates, 
none of the two tests considered is always superior in the mixture alternative setting, 
and the  superiority depends on the nuisance parameters 
as depicted in figure~\ref{fig:Dominance-Region}. This is not surprising and serves as a reminder that the nature of 
the deviation from the null hypothesis has very concrete implications on the test to be performed. 

We have considered only the t-statistic and the signed-rank statistic since (a) these are the most common tests employed when testing for deviations from 
a single, symmetric, centred population and (b) a generalized likelihood ratio test for mixture alternatives is impractically hard and does 
not dominate other simple test statistics \citep{delmas_2003}.

\begin{figure}[h]
	\includegraphics[width=20pc]{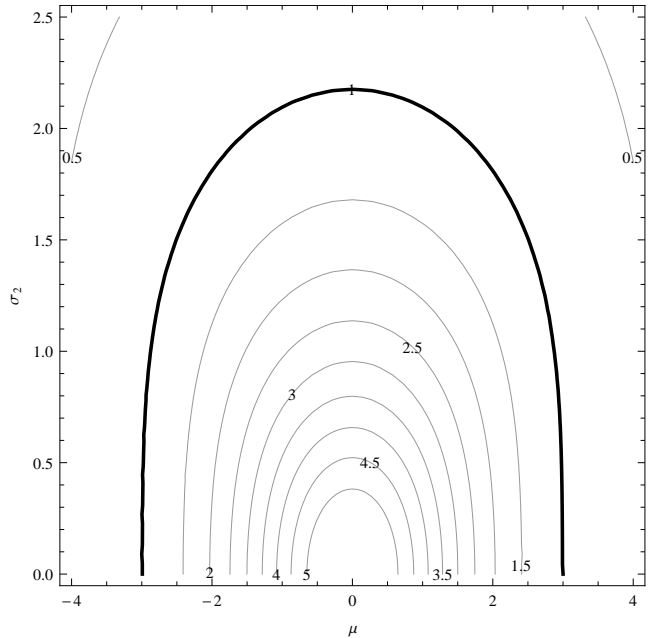}
	\caption{Asymptotic relative efficiency of the signed-rank test relative to the t-test as a function of 
	$\left(\mu,\sigma\right)$,
	when testing $\theta=0$ versus $\theta>0$ in 
	$\left(1-\theta\right)\mathcal{N}\left(0,1\right)+\theta\mathcal{N}\left(\mu,\sigma^{2}\right)$. }
	\label{fig:Dominance-Region}
\end{figure}

Pitman's asymptotic analysis is attractive since it often shows stability in finite and possibly small samples \citep{lehmann_parametric_2009}. We used a simulation study to analyze the finite sample power ratio of our two statistics. The asymptotic analysis in figure~\ref{fig:Dominance-Region} prefers the signed-rank test when the coefficient of variation is expected to be roughly smaller than 2/3. The finite sample analysis however,  tells a more complicated tale. Figures~\ref{fig:finite-sample-small-sd} and~\ref{fig:finite-sample-same-sd} demonstrate the power ratio between the two statistics as a function of the mixing proportion $\theta$, in samples of different sizes. Just like the asymptotic case, the dominance of the t statistic is not guaranteed. Its dominance region however, depends on both the nuisance parameters, including the mixing proportion,  and the sample size. Figure~\ref{fig:finite-sample-small-sd} depicts a scenario where the signed-rank statistic is considerably more powerful in a wide range of sample sizes and mixing proportions. The scenario in figure~\ref{fig:finite-sample-same-sd} is not as favourable for the signed-rank test, but does demonstrate that when dealing with a mixture of shifted Gaussians, the two statistics have remarkably similar performances with no single favourite.

\begin{figure}
	\textit{\includegraphics[width=20pc]{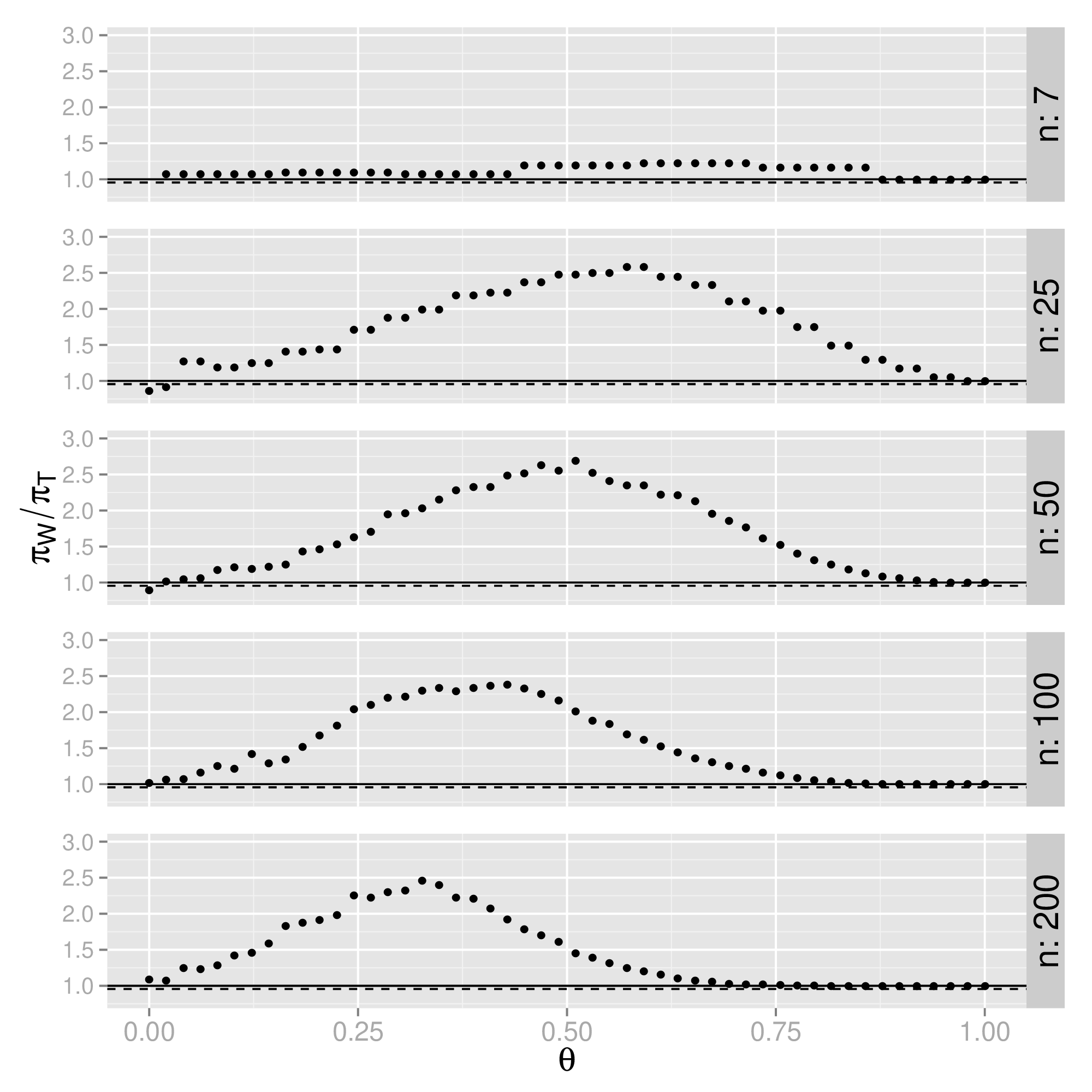}
	\caption{Power ratio of the signed-rank test (nominator) versus the t-test as a function of $\theta$ and the sample size,
	when testing $\theta=0$ versus $\theta>0$ in 
	$ \left( 1-\theta \right) \mathcal{N} \left( 0,1 \right) + \theta \mathcal{N} \left( 0.2, 0.1 \right) $. }
	\label{fig:finite-sample-small-sd}
}\end{figure}

\begin{figure}
	\includegraphics[width=20pc]{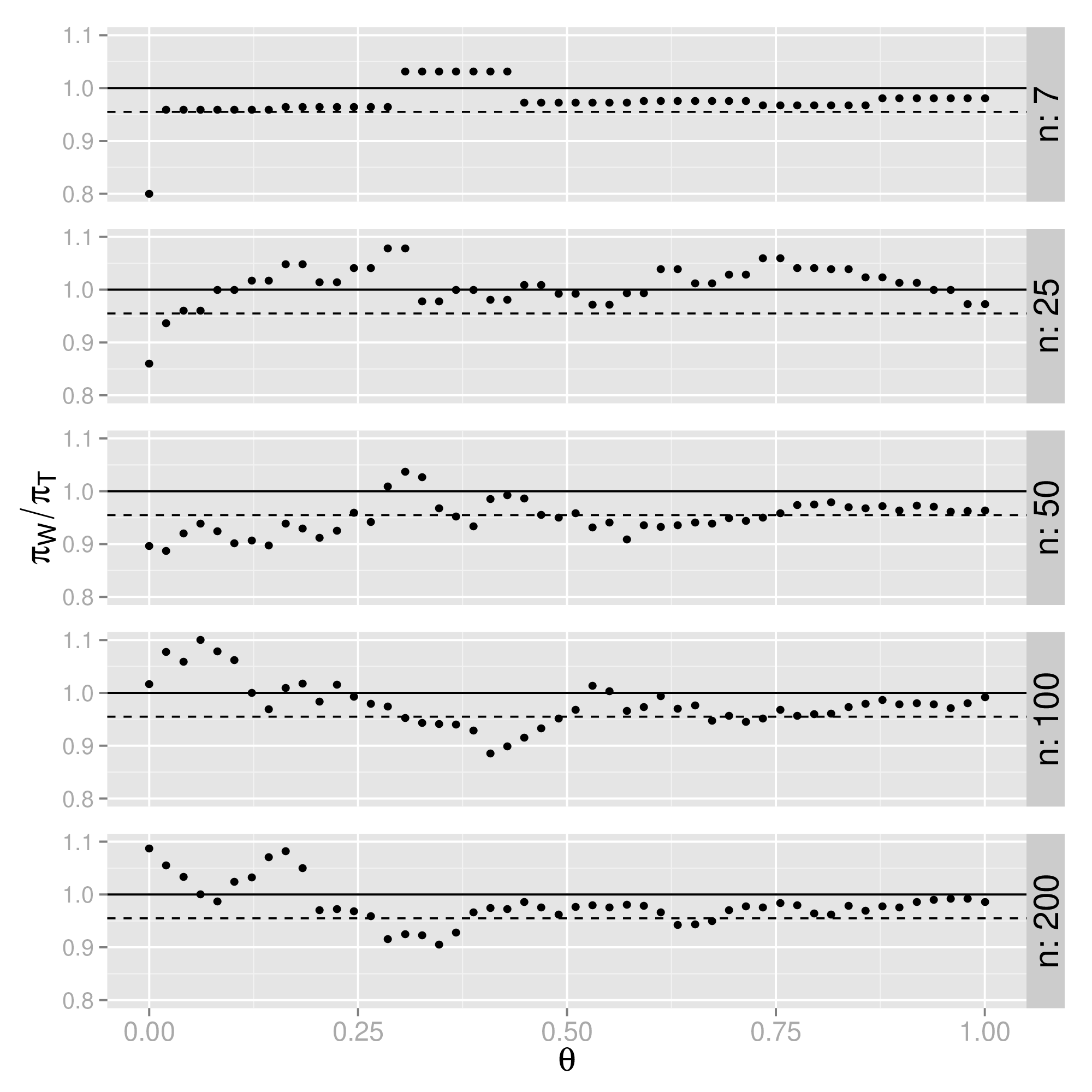}
	\caption{Power ratio of the signed-rank test (nominator) versus the t-test as a function of $\theta$ and the sample size,
	when testing $\theta=0$ versus $\theta>0$ in 
	$ \left( 1-\theta \right) \mathcal{N} \left( 0,1 \right) + \theta \mathcal{N} \left( 0.2, 1 \right) $. }
	\label{fig:finite-sample-same-sd}
\end{figure}

The results in this note have two practical implication on the choice of the test statistic when testing for deviations from a centred symmetric distribution. 
First, when the non parametric alternative has already been chosen, it suggests the power loss might not be too large. In particular when there is reason to believe the deviation from the null has the form of a mixture with a very concentrated non-centred component. 
Second, when the test statistic has not yet been chosen, it suggests that even under a Gaussian null, there is a possible power gain for the non-parametric statistic. In particular when the deviation from the null can be assumed to be a very asymmetric mixture. One which is obtained by mixing with a slightly-shifted and very centred component. Our suggested intuition
is that the t-test is more sensitive to the location shift
while the signed-rank statistic also captures the shape change occurring
as the null is mixed.

A possible application is the whole-brain functional magnetic resonance imaging  example. 
A much wider application can be found
in clinical trials. Consider the testing of a new drug: It is widely
acknowledged that the drug will affect only part of the population,
and yet it is common to test for a shift alternative. If one assumes
the drug has a shift effect on the affected sub-population, the mixture
alternative seems a more natural formulation. If the drug's expected
effect is small, the above result suggests the signed-rank test will
enjoy more power, even when the Gaussian-under-the-null assumption
holds.

Returning to the brain imaging example; After adapting Theorem~\ref{theorem1} to the appropriate mixture model, and using the nuisance parameter values estimated from that data, we find 
$e_{Wilcoxon,T} = 5.35$ 
suggesting that if the mixture alternative is of interest, the signed-rank statistic is indeed about five times more efficient.

\section{Acknowledgements}

Yoav Benjamini and Jonathan Rosenblatt were supported by a European
Research Council Advanced Investigator Grant (P.S.A.R.P.S.).

\bibliographystyle{plainnat}
\bibliography{Efficiency}

\end{document}